\documentclass[11pt,a4paper]{article}
\usepackage[dvips]{color}
\usepackage{indentfirst,latexsym,bm}
\usepackage{graphicx}
\usepackage{amsmath,amsthm}
\usepackage{amssymb}
\usepackage{multicol}
\usepackage{ulem}
\usepackage{amsfonts}
\usepackage{mathrsfs}
\usepackage[pdfstartview=FitH]{hyperref}
\usepackage{tikz}
\usetikzlibrary{trees,shapes,snakes,arrows,backgrounds}

\allowdisplaybreaks

\newtheorem{lem}{Lemma}[section]
\newtheorem{prop}[lem]{Proposition}
\newtheorem{thm}[lem]{Theorem}
\newtheorem{cor}[lem]{Corollary}

 \newtheorem{nott}{Notation}
\newtheorem{rem}{Remark}
\theoremstyle{definition}

 \DeclareMathAlphabet{\mathsfsl}{OT1}{cmss}{m}{sl}

 \newcommand{\Rnum}{\mathbb{R}}

 \newcommand{\Nnum}{\mathbb{N}}
 
 \newcommand{\dif}{\mathrm{d}}
  
 \newcommand{\tensor}[1]{\mathsf{#1}}
 
 \newcommand{\set}[1]{\left\{#1\right\}}
 
 \newcommand{\innp}[1]{\langle {#1}\rangle}

\numberwithin{equation}{section}
\begin{document}

\title{\bf{On the Jordan decomposition for a class of non-symmetric Ornstein-Uhlenbeck operators}}
\author{\rm\small
\noindent CHEN Yong\\
\noindent \footnotesize School of Mathematics and Computing Science, Hunan
University of Science and Technology,\\
\noindent \footnotesize Xiangtan, Hunan, 411201,
P.R.China. chenyong77@gmail.com\\
\rm\small \noindent LI Ying\\
\noindent \footnotesize  School of Mathematical Sciences,\ Beijing Normal University, \
\\
\noindent \footnotesize Beijing,  {\rm 100875}, P.R.China.
yingli618@gmail.com\\
}
\date{}
 \maketitle
\maketitle \noindent {\bf Abstract } \\
In this paper, we calculate the Jordan decomposition (or say, the Jordan canonical form) for a class of non-symmetric Ornstein-Uhlenbeck operators with the drift coefficient matrix being a Jordan block and the diffusion coefficient matrix being identity multiplying a constant. For the 2-dimensional case, we present all the general eigenfunctions by the induction. For the 3-dimensional case, we divide the calculating of the Jordan decomposition into several steps (the key step is to do the canonical projection onto the homogeneous Hermite polynomials, next we use the theory of systems of linear equations). As a by-pass product, we get the geometric multiplicity of the eigenvalue of the Ornstein-Uhlenbeck operator.

\maketitle

\section{Introduction}

We consider the d-dimensional ($d\geqslant 2$) nonsymmetric Ornstein-Uhlenbeck process
\begin{align}
\begin{bmatrix}
\dif X_1(t)\\ \dif X_2(t)\\ \dots \\ \dif X_d(t)\end{bmatrix}
&=\begin{bmatrix} -c& 1 &0& \dots &0 \\ 0 & -c  &1 &\dots &0\\ \dots& \dots &\dots& \dots &\dots \\0& 0 &0&\dots& -c \end{bmatrix}
\begin{bmatrix} X_1(t)\\ X_2(t)\\ \dots \\ X_d(t) \end{bmatrix} \dif t + \sqrt{2\sigma^2}
\begin{bmatrix} \dif B_1(t)\\ \dif B_2(t)\\ \dots \\ \dif B_d(t) \end{bmatrix} .\label{langevin}
\end{align}
The associated Ornstein-Uhlenbeck operator is
\begin{equation}\label{ou.op}
A_d = (-cx_1+ x_2)\frac{\partial}{\partial x_1}
+(-cx_2+x_3)\frac{\partial}{\partial x_3}+\cdots -cx_d\frac{\partial}{\partial x_d} +\sigma^2
(\frac{\partial^2}{\partial x_1^2}+\frac{\partial^2}{\partial x_2^2}+\cdots+\frac{\partial^2}{\partial x_d^2})\, .
\end{equation}
The associated Markov semigroup $(T(t))_{t\geqslant 0}$ on the Banach space of the bounded measurable functions is
\begin{equation}
 ( T(t)f) (x)=\frac{1}{(4\pi)^{3/2}(det\, \tensor{Q}_t)^{1/2}}\int_{\Rnum^3} e^{-\innp{\tensor{Q}_t^{-1} y, \, y}/4} f(e^{t\tensor{B}}x-y)\,\dif y,
\end{equation}
where $\tensor{B}=-c \mathrm{Id}+\tensor{R}$ with $\mathrm{Id}$ the identity and $\tensor{R}$ the nilpotent, and
\begin{equation*}
  \tensor{Q}_t=\sigma^2 \int_{0}^t e^{s\tensor{B}}e^{s\tensor{B}^*} \,\dif s
\end{equation*}
and $\tensor{B}^*$ denotes the transpose matrix of $\tensor{B} $.

It is well known that $(T(t))_{t\geqslant 0}$ extends to a strongly continuous semigroup of positive contractions in the Hilbert space $L^2_{\mu}=L^2(\Rnum^d,\,\dif \mu)$, where $\mu$ is the unique invariant measure \cite{cb,lun}.
Still denote by $(A_d,\,D)$ the generator of $(T(t))_{t\geqslant 0}$ in $L^2_{\mu}$, and it was shown \cite{Mg} that the spectrum consists of eigenvalues of finite multiplicities, $\sigma(A_d)=\set{-n c:\,n\in \Nnum}$ and all the generalised eigenfunctions are polynomials and form a complete system in $L^2_{\mu}$.
Let $\gamma=-nc$. It follows from \cite[Theorem 4.1]{Mg} that the algebraic multiplicity of $\gamma$ is
\begin{equation}\label{daishu}
   k_{_{A_d}}(\gamma)={n+d-1 \choose d-1},
\end{equation}
and it follows from \cite[Proposition 4.3, Theorem 4.1]{Mg} that $\nu_{_A}(\gamma)$, the index of the eigenvalue $\gamma$, is
\begin{equation}
  \nu_{_{A_d}}(\gamma)=1+(d-1)n.
\end{equation}

A natural question is what the geometric multiplicity of the eigenvalue $\gamma$ is. In addition, since the spectral subspace associated to $\gamma$ (i.e. $\mathrm{Ker}( \gamma -A_d)^{\nu_{_{A_d}}(\gamma)}$) is a finite-dimensional vector space over the real field $\Rnum$, what is the Jordan decomposition (or the Jordan canonical form) \cite{LA,mpj} of $A_d$ restricted on the spectral subspace?  That is to say, what are the integers $r>0,\, 0<q_r\leqslant q_{r-1}\leqslant \dots \leqslant q_1\leqslant q_0\leqslant \nu_{_A}(\gamma)$ and the generalised eigenfunctions $f_r, f_{r-1},\dots,f_1, f_0$, such that
\begin{equation}
  \set{f_k,( \gamma -A_d)f_k,\dots, ( \gamma -A_d)^{q_k-1}f_k:\,k=0,1,\dots,r }
\end{equation}
form the basis of the spectral subspace associated to $\gamma$, and
\begin{equation}\label{eig2}
   ( \gamma -A_d)^{q_k}f_k=0,\,k=0,1,\dots,r.
\end{equation}
The integers $(q_r,\,q_{r-1},\,\dots,q_1,\, q_0)$ are also called Segre characteristic (or say, Segre type, Segre notation). $f_k$ is called a lead vector (or say, a cyclic vector, a generator) of a Jordan chain $\set{f_k,( \gamma -A_d)f_k,\dots, ( \gamma -A_d)^{q_k-1} f_k }$ by some authors\cite{hk,mpj}.

In the present paper, we present an approach to calculate the Jordan decomposition and the generalised eigenfunctions (see Theorem~\ref{tm1},\ref{th1}) for $d=2,3$.\footnote{ We think that the approach is also able to solve the same problems for $d\geqslant4$.} The proof of Theorem~\ref{tm1} is by direct calculation. The main techniques of the proof of Theorem~\ref{th1} are canonical projection, the theory of systems of linear
equations. This approach is novel to the Jordan decomposition of differential operators as far as we know.

\section{In case of 2-dimension}\label{sec02}
In this section, we treat the case of $d=2$. Denote $\rho=\frac{\sigma^2}{c}$. The Hermite polynomials is defined by the formula $$H_n(x)=(-\rho)^n e^{x^2/2\rho}\frac{\dif^n}{\dif x^n}e^{-x^2/2\rho},\,n=1,2,\dots.$$
\begin{thm}\label{tm1}
The geometric multiplicity of the eigenvalue $\gamma$ is $1$. Set
\begin{equation}
    G_i(x)=\sum_{j=0}^{\lfloor i/2\rfloor}\,\frac{1}{2^j}\frac{1}{j!(i-2j)!}(-\frac{\rho}{2c^2})^jH_{i-2j}(x),\, i=0,1,\dots.
  \end{equation}
Suppose $f=G_n$, then $\set{f,\,(\gamma-A_2)f,\,(\gamma-A_2)^2 f,\dots,(\gamma-A_2)^n f}$ forms a basis of the spectral subspace associated to $\gamma$. And
\begin{equation}\label{fnk}
   (\gamma- A_2)^{k} f=(-1)^{k}\sum_{i=0\vee ( n- 2k)}^{n-k}\,(-\frac{\rho}{2c})^{n-k -i}{k \choose n-k -i}G_i(x)H_{2k- n+ i}(y),
  \end{equation}
where $a \vee b = \max\set{a,b}$ and $k=0,1,\dots,n$. Especially, $(\gamma- A_2)^{n} f=(-1)^n H_n(y)$ is the eigenfunction associated to $\gamma$.
\end{thm}
\begin{proof} we need only prove Eq.(\ref{fnk}).
  It is easy to check that $G_i(x)$ satisfies the recursion relation:
  \begin{equation}\label{a23}
   \left\{
      \begin{array}{ll}
      (- i c-A_2 )G_i(x)=-y C_{i-1}(x)+ \frac{\rho}{2c} G_{i-2}(x),  \\
      G_0(x)=1,\, G_1(x)=x.
      \end{array}
\right.
      \end{equation}
Clearly, for any differentiable function $h(x),g(y)$, we have
   \begin{align*}
     (\gamma -A_2 )h(x)g(y)=g(y )(\gamma -A_2 )h(x)+c\cdot h(x)(-\rho\frac{\partial^2}{\partial y^2}+ y\frac{\partial}{\partial y})g(y ).
   \end{align*}
Then by the property of the Hermite polynomials \cite{guo}, we have
   \begin{align*}
     (\gamma -A_2 )G_i(x)H_{2k- n+ i}(y)
     &=H_{2k- n+ i}(y)[-y G_{i-1}(x)+ \frac{\rho}{2c}G_{i-2}(x)+(i -n)c G_i(x)]\\
     &+ cG_i(x)( 2k -n+ i)H_{2k- n+ i}(y)\quad \text{(by (\ref{a23})}\\
     &= H_{2k- n+ i}(y)[-y G_{i-1}(x)+ \frac{\rho}{2c} G_{i-2}(x)+2(i+k-n)c G_i(x) ]\\
     &=-G_{i-1}(x)[H_{2k -n+ i+1}(y)+(2k-n+ i)\rho H_{2k- n+ i-1}(y) ]\\
     & + H_{2k -n+ i}(y)[\frac{\rho}{2c} G_{i-2}(x)+2(i +k-n)c G_i(x) ].
\end{align*}
When $n\geqslant 2k+2 $, by the mathematical induction, we have that
   \begin{align*}
      (-1)^{k+1}(\gamma- A_2)^{k+1}f&=-\sum _{i=0}^{n-k}\,(-\frac{\rho}{2c})^{n-k -i}{k \choose n-k -i}(\gamma- A_2) G_i(x)H_{2k-n+ i}(y)\\
      &= \sum_{i=1}^{n-k}\, (-\frac{\rho}{2c})^{n-k -i}{k \choose n -k -i }G_{i-1}(x)[H_{2k-n+ i+1}(y)+(2k-n+ i)\rho H_{2k-n+ i-1}(y) ]\\
      &+ \sum_{i=2}^{n-k}\, (-\frac{\rho}{2c})^{n-k+1 -i}{k \choose n - k -i }G_{i-2}(x)H_{2k-n+ i}(y)\\
      &-\rho \sum_{i=0}^{n-k-1}\, (-\frac{\rho}{2c})^{n-k -1-i}(n -i -k){k \choose n -k -i }G_{i}(x)H_{2k-n+ i}(y)\\
     &= \sum_{i=0}^{n-k-1}\, (-\frac{\rho}{2c})^{n-k -i-1}{k+1 \choose n-k -i-1}G_i(x)H_{2(k+1)- n+ i}(y).
   \end{align*}
 When $n< 2k+2 $, the proof is similar.
\end{proof}

\section {In case of 3-dimension}\label{sec2}
In this section, we treat the case of $d=3$. For convenience, we give some notations firstly.
Set $\mathcal{P}$ the space of all polynomials with variables (x,y,z), $\mathcal{P}_n$ the space of polynomials of degree less than or equal to $n$ and $\mathcal{H}_n$ the space of homogeneous polynomials of degree $n$. Then $\mathcal{P}=\cup_n \mathcal{P}_n$ and one has a usual direct sum decomposition of all polynomials\cite{Mg},
\begin{equation}
  \mathcal{P}_n=\oplus_{m=0}^n \mathcal{H}_m.
\end{equation}

\begin{nott}
 Set  $\gamma=-nc$, $r=\lfloor \frac{n}{2}\rfloor$ and $\rho=\frac{\sigma^2}{c}$.
\end{nott}

By the monomials property of the Hermite polynomials \cite{guo}
$$ x^n=\sum^{\lfloor n/2\rfloor}_{k=0}{n \choose 2k}(2k-1)!! \rho^k H_{n-2k}(x),$$
 the Hermite polynomials are another basis of $\mathcal{P}$.

 Let $\mathcal{H}_m'=\mathrm{span}\set{H_i(x)H_j(y)H_k(z),\,i+j+k=m}$, then we have another direct sum decomposition of all polynomials, i.e.,
 \begin{equation}
 \mathcal{P}=\cup_n \mathcal{P}_n,\quad  \mathcal{P}_n=\oplus_{m=0}^n \mathcal{H}_m'.
\end{equation}
We denote by $\mathcal{Q}_m$ the canonical projection \cite{Rs} of $\mathcal{P}$ onto $\mathcal{H}_m'$ .

\begin{thm}\label{th1} Let
  \begin{equation}
    q_k=2n+1-4k,\quad k=0,1,2,\dots,r.
  \end{equation}
Then there exist $\set{ f_k,\,k=0,\dots,r}$ so that $\set{f_k,( \gamma -A_3)f_k,\dots, ( \gamma -A_3)^{q_k-1}f_k:\,k=0,\dots,r }$ forms a basis of the spectral subspace associated to $\gamma$. Set $h_k=( \gamma -A_3)^{q_{_{k}}-1} f_{k}$. Then $\set{h_k,\,k=0,1,2,\dots,r}$ are  the basis of the eigenspace of the eigenvalue $\gamma$  and satisfy
\begin{equation}\label{eqq}
   \mathcal{Q}_n h_{k} =\sum\limits_{i=0}^{k}\,(-2)^{k-i} {k \choose i}H_{k-i}(x)H_{2i}(y)H_{n-k-i}(z).
\end{equation}
\end{thm}
Proof of Theorem~\ref{th1} is presented in Subsection~\ref{ssec}. The following is a by-pass product.
\begin{cor}\label{corr1}
 The geometric multiplicity of the eigenvalue $\gamma$ of the Ornstein-Uhlenbeck operator $A_3$ is $r+1$.
\end{cor}


\subsection{Proof of the theorem}\label{ssec}
  Note that $$A_3=-c(-\rho \frac{\partial^2 }{\partial x^2}+ x\frac{\partial}{\partial x})
-c(-\rho\frac{\partial^2}{\partial y^2}+ y\frac{\partial}{\partial y})-c(-\rho\frac{\partial^2}{\partial z^2}+ z\frac{\partial}{\partial z}) +y\frac{\partial}{\partial x}+z\frac{\partial}{\partial y}.$$
It follows from the property of the Hermite polynomials \cite{guo} that
 \begin{align}
   (\gamma-A_3 )(H_i(x)H_j(y)H_k(z))&= (m-n)c H_i(x)H_j(y)H_k(z)-iH_{i-1}(x)[H_{j+1}(y)+j\rho H_{j-1}(y)]H_k(z)\nonumber \\
   & - jH_{i}(x)H_{j-1}(y)[H_{k+1}(z)+k\rho H_{k-1}(z) ].\label{guu}
 \end{align}
For convenience, Eq.(\ref{guu}) can be rewritten as the following style.
\begin{prop}
If $\varphi \in \mathcal{H}_m'$ then $ (\gamma-A_3 )\varphi =\mathcal{Q}_m(\gamma-A_3 )\varphi + \mathcal{Q}_{m-2}(\gamma-A_3 )\varphi$. Especially,
 \begin{eqnarray}\label{eq0}
  (\gamma-A_3 )(H_i(x)H_j(y)H_k(z))& =& \mathcal{Q}_m(\gamma-A_3 )(H_i(x)H_j(y)H_k(z))\nonumber \\
  & +& \mathcal{Q}_{m-2}(\gamma-A_3 )(H_i(x)H_j(y)H_k(z))
 \end{eqnarray}
 where $m=i+j+k$ and
  \begin{eqnarray}\label{eq1}
     \mathcal{Q}_m(\gamma-A_3 )(H_i(x)H_j(y)H_k(z))&=&(m-n)c H_i(x)H_j(y)H_k(z) -iH_{i-1}(x)H_{j+1}(y)H_k(z)\nonumber \\
     &-& jH_{i}(x)H_{j-1}(y)H_{k+1}(z),
  \end{eqnarray}
  and
  \begin{equation}
     \mathcal{Q}_{m-2}(\gamma-A_3 )(H_i(x)H_j(y)H_k(z))=-ij\rho H_{i-1}(x)H_{j-1}(y)H_k(z)-jk\rho H_{i}(x)H_{j-1}(y)H_{k-1}(z).
  \end{equation}
  Especially, if $m=n$, then
  \begin{equation}\label{eq2}
     \mathcal{Q}_n(\gamma-A_3 )(H_i(x)H_j(y)H_k(z))=-iH_{i-1}(x)H_{j+1}(y)H_k(z)-jH_{i}(x)H_{j-1}(y)H_{k+1}(z).
  \end{equation}
\end{prop}
\begin{rem}\label{ideea}
  Eq.(\ref{eq1})---(\ref{eq2}) make us use the terminology of graph theory. In fact, by Eq.(\ref{eq2}), we have a weighted and directed acyclic graph(which also can be seen as a Hasse diagram) of the evolution of basis of $\mathcal{H}_n'$ operated by $ \mathcal{Q}_n(\gamma-A_3 )$. For example, when $n=3$ and $n=4$,\footnote{  The directed acyclic graph can be easily given for any integer $n$, and we ignore it for convenience of typing.} the directed acyclic graphs are respectively
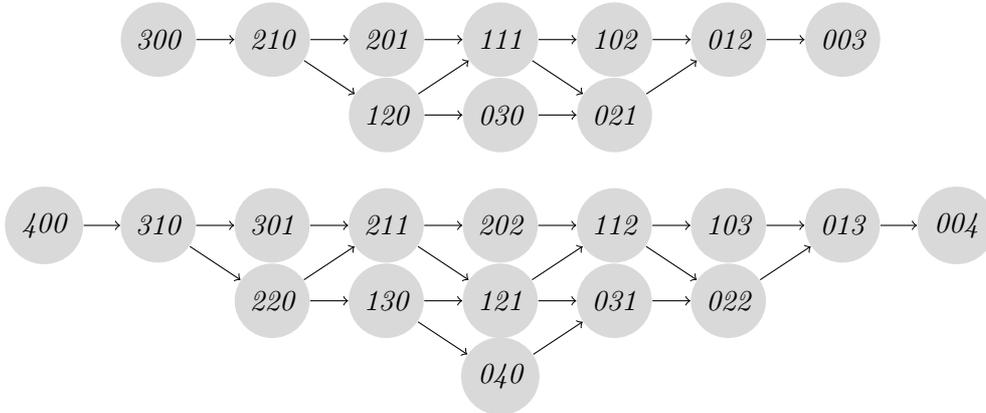
\begin{center}
\begin{tikzpicture}[scale=1, transform shape]
\tikzstyle{every node} = [circle, fill=gray!30]
\node (a) at (0, 1) {300};
\node (b) at +(1.5, 1) {210};
\node (c) at +(3, 0) {120};
\node (d) at +(3, 1) {201};
\node (e) at +(4.5,0) {030};
\node (f) at +(4.5,1) {111};
\node (g) at +(6,0) {021};
\node (h) at +(6,1) {102};
\node (i) at +(7.5,1) {012};
\node (j) at +(9, 1) {003};
\foreach \from/\to in {a/b, b/c, b/d, c/e, c/f, d/f, e/g,f/g,f/h,g/i,h/i,i/j}
\draw [->] (\from) -- (\to);
\end{tikzpicture}
\end{center}
\begin{center}
\begin{tikzpicture}[scale=1, transform shape]
\tikzstyle{every node} = [circle, fill=gray!30]
\node (a) at (0, 1) {400};
\node (b) at +(1.5, 1) {310};
\node (c) at +(3, 0) {220};
\node (d) at +(3, 1) {301};
\node (e) at +(4.5, 0) {130};
\node (f) at +(4.5,1) {211};
\node (g) at +(6,-1) {040};
\node (h) at +(6,0) {121};
\node (h0) at +(6,1) {202};
\node (i) at +(7.5,0) {031};
\node (i0) at +(7.5,1) {112};
\node (j) at +(9, 0) {022};
\node (j0) at +(9,1) {103};
\node (k) at +(10.5,1) {013};
\node (m) at +(12,1) {004};
\foreach \from/\to in {a/b, b/c, b/d, c/e, c/f, d/f, e/g,e/h, f/h, f/h0, g/i, h/i, h/i0, h0/i0, i/j, i0/j, i0/j0, j/k, j0/k, k/m}
\draw [->] (\from) -- (\to);
\end{tikzpicture}
\end{center}
where we denote by the triple integers $(i,j,k) $ the Hermite polynomial $H_i(x)H_j(y)H_k(z)$. It follows from Eq.(\ref{eq2}) that the weights of the arrows between $(i,j,k)$ and $(i,,j-1,k+1),\,(i-1,j+1,k)$ are $-j,\,-i$ respectively.

  One can find out many properties from the directed acyclic graph. For example, for the vertex $(i,j,k)$ with $i+j+k=n$, the height (which is defined by the distance between vertices $(n,0,0)$ and $(i,j,k)$, and thus is between $0\sim 2n$) is $h=j+2k$. For simplicity, in each height of the graph, we list the vertices $(i,j,k)$ decreasingly by lexicographic order. Then the vertices $(i,j,k)$ and $(k,j,i)$ are symmetric about the $n$-th height of the graph.

It follows from Eq.(\ref{daishu}), Corollary~\ref{corr1} and Theorem~\ref{th1} that the order of the graph (the numbers of vertices in the graph) is the algebraic multiplicity of $\gamma$, $q_k$ are $1$ plus the distance between vertices $(\lceil\frac{n}{2}\rceil +k,0,\lceil\frac{n}{2}\rceil-k)$ and $(\lceil\frac{n}{2}\rceil-k,0,\lceil\frac{n}{2}\rceil+k)$, and the numbers of vertices in the $n$-th height of the graph is the geometric multiplicity of $\gamma$.
\end{rem}


\begin{prop}\label{pp33}
The spectral subspace associated to $\gamma=-nc$ belongs to $\oplus_{i=0}^{\lfloor \frac{n}{2}\rfloor} \mathcal{H}_{n-2i}'$.
\end{prop}
\begin{proof}
Suppose that $f$ is a generalised eigenfunction, i.e., there exists a integer $k \geqslant 1$ such that $(\lambda-A_3)^k f=0$. It follows from Eq.(\ref{eq1}) that if the degree of $f$ is $m\neq n$ then $\mathcal{Q}_m(\lambda-A_3)^k f\neq 0$. This is a contradiction, then $f\in \mathcal{P}_n$ and the degree is exactly $n$.\footnote{The reader can also refer to  \cite[Proposition 3.1]{Mg}.}
If there is an $i=0,1,\dots,\lfloor\frac{n-1}{2}\rfloor$ such that $\mathcal{Q}_{n-1-2i}f\neq 0$, then by Eq.(\ref{eq0}), $\mathcal{Q}_{n-1-2i}(\lambda-A_3)^k f\neq 0$. This is a contradiction, then $f\in \oplus_{i=0}^{\lfloor \frac{n}{2}\rfloor} \mathcal{H}_{n-2i}'$.
\end{proof}
\begin{lem}\label{pp1}
 For any polynomial $ g\in \mathcal{H}_{m}'$ with $m\neq n$, there exists a unique solution $f\in \mathcal{H}_m'$ to the equation $\mathcal{Q}_m(\gamma-A_3)f=g$.
\end{lem}
\begin{proof}
Suppose that $g=\sum\limits_{i+j+k=m} b_{ijk}H_i(x)H_j(y)H_k(z)$ and $f=\sum\limits_{i+j+k=m} a_{ijk}H_i(x)H_j(y)H_k(z)$. It follows from Eq.(\ref{eq1}) that
\begin{align*}
  & \mathcal{Q}_m(\gamma-A_3)f \\&= \mathcal{Q}_m(\gamma-A_3)\sum_{i+j+k=m} a_{ijk}H_i(x)H_j(y)H_k(z)\\
   &= \sum_{i+j+k=m} a_{ijk}[(m-n)c H_i(x)H_j(y)H_k(z) - i  H_{i-1}(x)H_{j+1}(y)H_k(z)-j H_i(x)H_{j-1}(y)H_{k+1}(z)]\\
   &=\sum\limits_{i+j+k=m} b_{ijk} H_i(x)H_j(y)H_k(z).
\end{align*}
By the linear independent of $\set{H_i(x)H_j(y)H_k(z):\,i+j+k=m}$, we have a system of ${m+1\choose 2 }$ linear equations in ${m+1\choose 2 }$ unknowns.

If we sort $\set{H_i(x)H_j(y)H_k(z):\,i+j+k=m}$ by which appear in the directed acyclic graph in Remark~\ref{ideea}. The coefficient matrix of the linear equations is a lower triangle matrix with nonzero diagonal entry $(m-n)c$. Thus the linear equations have a unique solution.
\end{proof}


Suppose that $f=\sum_{s=0}^{\lfloor \frac{m}{2}\rfloor}f_{m-2s} $ with $f_{s}\in \mathcal{H}_s'$. It follows from Eq.(\ref{eq0}) that
   \begin{eqnarray*}
     (\gamma-A_3)f &=& (\gamma-A_3)\sum_{s=0}^{\lfloor \frac{m}{2}\rfloor}f_{m-2s}  \\
     &=&\sum_{s=0}^{\lfloor \frac{m}{2}\rfloor}[\mathcal{Q}_{m-2s}(\gamma-A_3)f_{m-2s} +\mathcal{Q}_{m-2s-2}(\gamma-A_3)f_{m-2s}]\\
     &=& \mathcal{Q}_{m}(\gamma-A_3)f_{m}+ \sum_{s=1}^{\lfloor \frac{m}{2}\rfloor}\mathcal{Q}_{m-2s}(\gamma-A_3)[f_{m-2s} +f_{m+2-2s}].
   \end{eqnarray*}
   Thus the equation $(\gamma-A_3)f=g$ is equal to a system of equations:
   \begin{eqnarray}
     \mathcal{Q}_{m}(\gamma-A_3)f_{m}&=& g,\label{a1}\\
     \mathcal{Q}_{m-2s}(\gamma-A_3)f_{m-2s}&=&- \mathcal{Q}_{m-2s}(\gamma-A)f_{m+2-2s},\quad s=1,2,\dots,\lfloor \frac{m}{2}\rfloor.\label{a2}
   \end{eqnarray}
   It follows from Lemma~\ref{pp1} that when $m\neq n$, Eq.(\ref{a1}) has a unique solution $f_m\in \mathcal{H}'_m$ and  when $m-2s \neq n$, Eq.(\ref{a2}) has a unique solution $f_{m-2s}\in \mathcal{H}'_{m-2s}$.

 Clearly, if $f$ satisfies $(\gamma-A_3)f=g$, so does $f+h$ where $h$ is any eigenfunction of $A_3$ associated to $\gamma$. Thus we have the following Proposition.
\begin{prop}\label{pp2}
 For any $g\in \mathcal{P}$ with $\mathcal{Q}_{n+2s} g=0,\, s=0,1,2,\dots$, there exist solutions $f\in \mathcal{P}$ to the equation $(\gamma-A_3)f=g$. In addition, if $f$ is the same degree polynomial to $g$ then there exists one and only one solution.
\end{prop}

\begin{prop}\label{ll3}
Set $\psi_k=\sum\limits_{i=0}^{k}\, (-2)^{k-i} {k \choose i} H_{k-i}(x)H_{2i}(y)H_{n-k-i}(z),\,k=0,1,\cdots,r$. Then it satisfies the equation $\mathcal{Q}_n(\gamma-A_3)\psi_k=0$.
\end{prop}
\begin{proof}
Suppose that $\psi_k=\sum\limits_{i=0}^{k}\,a_{i}  H_{k-i}(x)H_{2i}(y)H_{n-k-i}(z)$. By Eq.(\ref{eq1}), we have that
\begin{align*}
   \mathcal{Q}_n(\gamma-A_3)\psi_k &= \mathcal{Q}_n(\gamma-A_3)\sum_{i=0}^{k}\,a_{i}  H_{k-i}(x)H_{2i}(y)H_{n-k-i}(z) \nonumber \\
   &= \sum_{i=0}^k - a_{i}[(k-i)  H_{k-i-1}(x)H_{2i+1}(y)H_{n-k-i}(z)+2i  H_{k-i}(x)H_{2i-1}(y)H_{n-k-i+1}(z)]\nonumber\\
   &= 0. \label{ll2}
\end{align*}
By the linear independent of $\set{H_i(x)H_j(y)H_k(z)}$, we have a system of $k $ linear homogeneous equations in $k+1 $ unknowns. The coefficient matrix is
  \begin{equation*}
\tensor{M}_k= - \left[
\begin{array}{llllllll}
     k & 2 &0    & 0   & \dots  & 0 & 0& 0\\
     0 & k-1 &4  & 0   & \dots  & 0 & 0& 0\\
     0   & 0 & k-2 &6  & \dots  & 0 & 0& 0\\
     \hdotsfor{8}\\
     0   & 0   & 0   &0    & \dots  &2&2(k-1) &0\\
     0   & 0   & 0   &0    & \dots  &0 &1 &  2k
\end{array}
\right ].
\end{equation*}
Clearly, the solution is 1-dimension and $a_i=(-2)^{k-i} {k \choose i},\,i=0,\dots,k$ is a solution.
\end{proof}
Now suppose that $h=\psi+\phi$ with $\psi\in \mathcal{H}_n'$ and $\phi\in \mathcal{P}_{n-2}$, then by Eq.(\ref{eq0}),
\begin{eqnarray*}
   (\gamma-A_3)h &=&(\gamma-A_3)\psi +(\gamma-A_3)\phi\\
   &=& \mathcal{Q}_n(\gamma-A_3)\psi+[ \mathcal{Q}_{n-2}(\gamma-A_3)\psi + (\gamma-A_3)\phi].
\end{eqnarray*}
Therefore, the equation $(\gamma-A_3)h=0$ is equivalent to two equations:
\begin{eqnarray}
   \mathcal{Q}_n(\gamma-A_3)\psi&=&0,\quad \psi\in \mathcal{H}_n',\label{c1}\\
   (\gamma-A_3)\phi&=&- \mathcal{Q}_{n-2}(\gamma-A_3)\psi,\quad \phi\in \mathcal{P}_{n-2}.\label{c2}
 \end{eqnarray}
By Proposition~\ref{ll3}, Eq.(\ref{c1}) has $1+r$ independent solutions. It follows from Proposition~\ref{pp2} that there exists unique $\phi$ satisfying Eq.(\ref{c2}) for each $\psi$.  Thus we have the following corollary.
\begin{cor}\label{pp7}
The geometric multiplicity of the eigenvalue $\gamma=-nc$ is greater than or equal to $1+r$ (i.e., there are at least $1+r$ independent solutions to the equation $(\gamma-A_3)h=0$).
\end{cor}

Denote by $h_k,\,k=0,1,\dots,r$ the solutions to the equation $(\gamma-A_3)h=0$ given by Corollary~\ref{pp7}. Clearly, $h_k= \psi_k +(Id- \mathcal{Q}_n)h_k$, where $\psi_k\in \mathcal{H}_n'$ is the same as in Proposition~\ref{ll3}.

Suppose that $f_k=\varphi_k+g_k$ with $\varphi_k\in \mathcal{H}_n'$ and $g_k\in \mathcal{P}_{n-2}$, then
\begin{eqnarray*}
   (\gamma-A_3)^{q_k-1}f_k&=& (\gamma-A_3)^{q_k-1}(\varphi_k+g_k)\\
   &=& \mathcal{Q}_n(\gamma-A_3)^{q_k-1}\varphi_k +  \mathcal{Q}_{n-2}(\gamma-A_3)^{q_k-1}\varphi_k+(\gamma-A_3)^{q_k-1} g_k.
\end{eqnarray*}
Therefore, the equation $(\gamma-A_3)^{q_k-1}f_k=h_k$ is equivalent to two equations:
\begin{eqnarray}
  \mathcal{Q}_n(\gamma-A_3)^{q_k-1}\varphi_k&=& \psi_k,\label{ee00}\\
 (\gamma-A_3)^{q_k-1} g_k&=&(Id- \mathcal{Q}_n)h_k- \mathcal{Q}_{n-2}(\gamma-A_3)^{q_k-1}\varphi_k.\label{ee11}
\end{eqnarray}
 Note that $\psi_k=\sum\limits_{i=0}^{k}b_i  H_{k-i}(x)H_{2i}(y)H_{n-k-i}(z)$. Set $\varphi_k=\sum\limits_{i=0}^{k}a_i H_{n-k-i}(x)H_{2i}(y)H_{k-i}(z)$, then l.h.s. of Eq.(\ref{ee00}) is a linear mapping from $\mathrm{span}\set { H_{n-k-i}(x)H_{2i}(y)H_{k-i}(z):\,i=0,\cdots, k}$ to $\mathrm{span}\set{H_{k-i}(x)H_{2i}(y)H_{n-k-i}(z):\,i=0,\cdots, k}$. The linear mapping is the evolution from the $2k$-height to the $2(n-k)$-height of the directed acyclic graph in Remark~\ref{ideea}, which is represented under the natural basis by a $(k+1)$-square matrix $\tensor{S}_{r-k}$ ( it is the multiplication of some matrices, for details, please refer to Subsection~\ref{ssec2}). By Proposition~\ref{pp20}, the matrix $\tensor{S}_{r-k}$ is nonsingular, which implies that Eq.(\ref{ee00}) has a solution. Since $g_k,\, (Id- \mathcal{Q}_n)h_k \in \mathcal{P}_{n-2}$, it follows from Proposition~\ref{pp2} that Eq.(\ref{ee11}) has a solution. Then we have the following Proposition.

\begin{prop}\label{pp9}
There exists an $f_k\in \mathcal{P}_n$ such that $(\gamma-A)^{q_k-1}f_k=h_k$.
\end{prop}

\noindent{\it Proof of Theorem \ref{th1}.\,}
 Note that $\psi_k=\mathcal{Q}_n h_k$ in Proposition~\ref{ll3} are linear independent, so does the eigenfunctions $h_k$. Let $\set{f_k,\,k=0,1,\cdots, r}$ be as in Proposition~\ref{pp9}. Then the generalised eigenfunctions $\set{(\gamma-A)^j f_k:\,j=0,1,\cdots,q_k-1,\,k=0,1,\cdots,r}$ are linear independent (please refer to the proof of \cite[pp264, Theorem 6.2.]{ls}).

Note that $q_k=2n+1-4k$, then the algebraic multiplicity of the eigenvalue $\gamma$ can be decomposed to
\begin{equation}\label{decom1}
  {n+2 \choose 2}=\sum_{k=0}^{r} (2n+1-4k)=\sum_{k=0}^{r} q_k.
\end{equation}
Thus $\set{(\gamma-A)^j f_k:\,j=0,1,\cdots,q_k-1,\,k=0,1,\cdots,r}$ forms a basis of the spectral subspace associated to $\gamma$.
Together with Proposition~\ref{pp7}, we have that the geometric multiplicity of the eigenvalue $\gamma$ is equal to $r+1$ (Otherwise, the algebraic multiplicity should be greater than ${n+1 \choose 2}$). Then $\set{h_k,\,k=0,\dots, r}$ forms the basis of the eigenspace of $\gamma$. Eq.(\ref{eqq}) is exactly
 the conclusion of Proposition~\ref{ll3}.

{\hfill\large{$\Box$}}

\subsection{Linear mapping represented by the multiplication of some matrices.}\label{ssec2}
For example, by Eq.(\ref{eq2}), when $n$ is odd, the evolution from the $(n-1)$-height to the $n$-height (the $n$-height to the $(n+1)$-height) of the directed acyclic graph is
\begin{align*}
& \mathcal{Q}_n(\gamma-A_3)\sum_{i=0}^{r}a_i H_{r+1-i}(x)H_{2i}(y)H_{r-i}(z)\\
   &=\sum_{i=0}^{r}a_i( -(r+1-i)  H_{r-i}(x)H_{1+2i}(y)H_{r-i}(z)- 2i H_{r+1-i}(x)H_{2i-1}(y)H_{r+ 1-i}(z))\\
   &=\sum_{i=0}^{r}( -(r+1-i)a_{i} - (2+2i)a_{i+1} )H_{r-i}(x)H_{1+2i}(y)H_{r-i}(z),\quad (\text { where $a_{r+1}=0$})\\
& \mathcal{Q}_n(\gamma-A_3)\sum_{i=0}^{r}b_i H_{r-i}(x)H_{1+2i}(y)H_{r-i}(z)\\
   &=\sum_{i=0}^{r}( -(r+1-i)b_{i-1} - (1+2i)b_i )H_{r-i}(x)H_{2i}(y)H_{r+ 1-i}(z).\quad (\text { where $b_{-1}=0$})
\end{align*}
Then the matrices associated to the linear mappings are $-\tensor{D}_0, -\tensor{A}_0$ (see below). When $n$ is even, the evolution from the $(n-1)$-height to the $n$-height (the $n$-height to the $(n+1)$-height) of the directed acyclic graph is
\begin{align*}
& \mathcal{Q}_n(\gamma-A_3)\sum_{i=0}^{r-1}c_i H_{r-i}(x)H_{2i+1}(y)H_{r-1-i}(z)\\
   &=\sum_{i=0}^{r}( -(r+1-i)c_{i-1} - (1+2i)c_{i} )H_{r-i}(x)H_{2i}(y)H_{r-i}(z), \quad (\text { where $c_{-1}=c_r=0$})\\
& \mathcal{Q}_n(\gamma-A_3)\sum_{i=0}^{r}d_i H_{r-i}(x)H_{2i}(y)H_{r-i}(z)\\
   &=\sum_{i=0}^{r-1}( -(r-i)d_i - (2+2i)d_{i+1} ) H_{r-i-1}(x)H_{2i+1}(y)H_{r-i}(z).
\end{align*}
Then the matrices associated to the linear mappings are $-\tensor{B}_1, -\tensor{C}_1$ (see below).

The others are similar. In general, the matrix $\tensor{S}_{r-k}$ (see Proposition~\ref{pp9}) associated to the linear mapping of the evolution from the $(2k)$-height to the $2(n-k)$-height of the directed acyclic graph is:
\begin{prop}\label{pp20}
 Let $r=\lfloor \frac{n}{2}\rfloor$. If $n$ is odd, then suppose that
$\tensor{S}_k$ is an $(r+1-k)$-square matrix given by
\begin{equation}\label{sk1}
   \left\{
      \begin{array}{ll}
  \tensor{S}_0&=\tensor{D}_0\tensor{A}_0,\\
    \tensor{S}_{k}&=\tensor{D}_k\tensor{C}_k\tensor{S}_{k-1}\tensor{B}_k\tensor{A}_k,
    \quad k=1,2,\dots,r,
       \end{array}
\right.
  \end{equation}
where $\tensor{D}_k, \tensor{A}_k$ are $(r+1-k)$-square matrixes,
      \begin{equation}
\tensor{A}_k= \left[
\begin{array}{llllllll}
     r+k+1 & 2 &0    & 0   & \dots  & 0 & 0& 0\\
     0 & r+k &4  & 0   & \dots  & 0 & 0& 0\\
     0   & 0 & r+k-1 &6  & \dots  & 0 & 0& 0\\
     \hdotsfor{8}\\
     0   & 0   & 0   &0    & \dots  &0&2k+2 &n- (2k+1)\\
     0   & 0   & 0   &0    & \dots  &0 &0  & 2k+1
\end{array}
\right ],
\end{equation}
and
  \begin{equation}
\tensor{D}_k= \left[
\begin{array}{llllllll}
     1 & 0 &0    & 0   & \dots  & 0 & 0& 0\\
     r-k & 3 &0  & 0   & \dots  & 0 & 0& 0\\
     0   & r-k-1 &5 &0  & \dots  & 0 & 0& 0\\
     \hdotsfor{8}\\
     0   & 0   & 0   &0    & \dots  &2&n-2k-2  &0\\
     0   & 0   & 0   &0    & \dots  &0 &1  & n-2k
\end{array}
\right ],
\end{equation}
and $\tensor{B}_k,\,\tensor{C}_k$ are $(r+2-k)\times (r+1-k),\,(r+1-k)\times (r+2-k)$ matrixes respectively,
  \begin{equation}
\tensor{B}_k= \left[
\begin{array}{llllllll}
     1 & 0 &0    & 0   & \dots  & 0 & 0& 0\\
     r+k & 3 &0  & 0   & \dots  & 0 & 0& 0\\
     0   & r+k-1 &5 &0  & \dots  & 0 & 0& 0\\
     \hdotsfor{8}\\
     0   & 0   & 0   &0    & \dots  &2k+2 &n-2k-2  &0\\
     0   & 0   & 0   &0    & \dots  &0 &2k+1  & n-2k\\
     0   & 0   & 0   &0    & \dots  &0 &0  & 2k\\
\end{array}
\right ],
\end{equation}
and
\begin{equation}
\tensor{C}_k= \left[
\begin{array}{llllllll}
     r-k+1 & 2 &0    & 0   & \dots  & 0 & 0& 0\\
     0 & r-k &4  & 0   & \dots  & 0 & 0& 0\\
     0   & 0 & r-k-1 &6  & \dots  & 0 & 0& 0\\
     \hdotsfor{8}\\
     0   & 0   & 0   &0    & \dots  &2& n- 2k-1 &0\\
     0   & 0   & 0   &0    & \dots  &0 &1 &  n- 2k+1
\end{array}
\right ].
\end{equation}
If $n$ is even, then suppose that $\tensor{S}_k$ is an $(r+1-k)$-square matrix given by
\begin{equation}\label{sk2}
   \left\{
      \begin{array}{ll}
      \tensor{S}_0&=\mathrm{Id}_{r+1},\\
      \tensor{S}_{k}&=\tensor{D}_k\tensor{C}_k\tensor{S}_{k-1}\tensor{B}_k\tensor{A}_k,
    \quad k=1,2,\dots,r,
     \end{array}
\right.
  \end{equation}
where $\tensor{D}_k, \tensor{A}_k$ are $r+1-k$ order matrixes,
      \begin{equation}
\tensor{A}_k= \left[
\begin{array}{llllllll}
     r+k & 2 &0    & 0   & \dots  & 0 & 0& 0\\
     0 & r+k-1 &4  & 0   & \dots  & 0 & 0& 0\\
     0   & 0 & r+k-2 &6  & \dots  & 0 & 0& 0\\
     \hdotsfor{8}\\
     0   & 0   & 0   &0    & \dots  &0&2k+1 &n- 2k\\
     0   & 0   & 0   &0    & \dots  &0 &0  & 2k
\end{array}
\right ],
\end{equation}
and
  \begin{equation}
\tensor{D}_k= \left[
\begin{array}{llllllll}
     1 & 0 &0    & 0   & \dots  & 0 & 0& 0\\
     r-k & 3 &0  & 0   & \dots  & 0 & 0& 0\\
     0   & r-k-1 &5 &0  & \dots  & 0 & 0& 0\\
     \hdotsfor{8}\\
     0   & 0   & 0   &0    & \dots  &2&n-2k-1  &0\\
     0   & 0   & 0   &0    & \dots  &0 &1  & n+1-2k
\end{array}
\right ],
\end{equation}
and $\tensor{B}_k,\,\tensor{C}_k$ are $(r+2-k)\times (r+1-k),\,(r+1-k)\times (r+2-k)$ matrixes respectively,
  \begin{equation}
\tensor{B}_k= \left[
\begin{array}{llllllll}
     1 & 0 &0    & 0   & \dots  & 0 & 0& 0\\
     r+k-1 & 3 &0  & 0   & \dots  & 0 & 0& 0\\
     0   & r+k-2 &5 &0  & \dots  & 0 & 0& 0\\
     \hdotsfor{8}\\
     0   & 0   & 0   &0    & \dots  &2k+2 &n-2k-1  &0\\
     0   & 0   & 0   &0    & \dots  &0 &2k  & n+1-2k\\
     0   & 0   & 0   &0    & \dots  &0 &0  & 2k-1\\
\end{array}
\right ],
\end{equation}
and
\begin{equation}
\tensor{C}_k= \left[
\begin{array}{llllllll}
     r-k+1 & 2 &0    & 0   & \dots  & 0 & 0& 0\\
     0 & r-k &4  & 0   & \dots  & 0 & 0& 0\\
     0   & 0 & r-k-1 &6  & \dots  & 0 & 0& 0\\
     \hdotsfor{8}\\
     0   & 0   & 0   &0    & \dots  &2& n- 2k&0\\
     0   & 0   & 0   &0    & \dots  &0 &1 &  n- 2k+2
\end{array}
\right ].
\end{equation}
Then the matrices $\tensor{S}_k$ are nonsingular.
\end{prop}
\begin{rem}
Set the column vector $\vec{u}_k=\big[(-2)^{k},\, {k \choose k-1}(-2)^{k-1},\,\dots, {k \choose 2}(-2)^{2},\,{k \choose 1}(-2) ,\,1\big]'$, where $k=0,1,\dots, r$.  We conjecture that the column vector $\vec{u}_{r-k}$ is the eigenvector of $\tensor{S}_k$ associated to the eigenvalue $\lambda_k$ which is defined by:\\
when $n$ is odd, \begin{equation}
   \left\{
      \begin{array}{ll}
      \lambda_0 &= 1,\\
      \lambda_{k} &= 2k (2k+1)(4k-1)(4k+1)\lambda_{k-1},\quad k=1,2,\dots, r .
      \end{array}
\right.
\end{equation}
when $n$ is even,
\begin{equation*}
   \left\{
      \begin{array}{ll}
      \lambda_0 &= 1,\\
      \lambda_{k} &= 2k (2k-1)(4k-3)(4k-1)\lambda_{k-1},\quad k=1,2,\dots, r .
      \end{array}
\right.
\end{equation*}
If the conjecture is valid, then we can characterize the leader vectors $f_k$ more clearly, i.e.,
\begin{equation*}
  \mathcal{Q}_n f_{k} = \sum\limits_{i=0}^{k}\,(-2)^{k-i} {k \choose i} H_{n-k-i}(x)H_{2i}(y)H_{k-i}(z).
\end{equation*}
\end{rem}
Apply the notation in \cite{hk,zfz}.
Let $\tensor{A}$ be an $p\times q$ matrix, $\alpha= \set{i_1, \dots, i_s}$, and $\beta = \set{j_1, . . . , j_t}$,
$1 \leqslant i_1 <\dots < i_s\leqslant p, 1 \leqslant j_1 <\dots < j_t \leqslant q$.
Denote by
$\tensor{A}[i_1,\dots , i_s; j_1,\dots, j_t]$ , or simply $\tensor{A}[\alpha, \beta]$, the submatrix of $\tensor{A}$ consisting
of the entries in rows $i_1, \dots , i_s$ and columns $j_1, \dots , j_t$. Let $\tensor{A}(i_1,\dots , i_s| j_1,\dots, j_t)$ be the submatrix of $\tensor{A}$ obtained by deleting rows $i_1, \dots , i_s$ and columns $j_1, \dots , j_t$. For convenience, $\tensor{A}(i_1,\dots , i_s| )$ ($\tensor{A}(| j_1,\dots, j_t)$) means by deleting rows (columns) only.
\begin{proof}
We divide the proof into three steps.\\
Claim 1: All the minors (i.e.,the determinant of the square submatrix )of the matrixes $\tensor{A_k},\,\tensor{B}_k,\,\tensor{C}_k,\,\tensor{D}_k,\,k=1,\dots,r,$ are nonnegative.
Since $\tensor{A}_k$ ($\tensor{B}_k$) and $\tensor{D}_k^{T}$ ($\tensor{C}_k^{T}$) have the same types, we only need to prove the cases of $\tensor{A}_k\, ,\tensor{B}_k$. The square submatrix of $\tensor{B}_k$ are the same to that of a certain $\tensor{B}_k(i|),\,i=1,\dots,r+2-k$ which is a direct sum of two matrices \cite{zfz} with the same type of $\tensor{D}_k$ or $\tensor{A}_k$.
Thus we only need to prove the case of $\tensor{A}_k$.
In fact, $\tensor{A}[i_1,\dots , i_s; j_1,\dots, j_s]$ is a 
direct sum \cite{zfz} of several nonnegative triangular matrices, which implies that $\tensor{A}[i_1,\dots , i_s; j_1,\dots, j_s]$ has nonnegative determinant.

Claim 2: All the minors of $\tensor{S}_k$ are nonnegative. In fact, we have the following type Binet-Cauchy formula:
\begin{align}\label{bc}
  \det (\tensor{AB})[\alpha,\beta]=\sum_{\kappa} \det \tensor{A}[\alpha,\kappa]\times \det \tensor{A}[\kappa, \beta],
\end{align}
where $\tensor{A}$ ($\tensor{B}$) is $m\times n$ ($n\times m$) matrix, $m\leqslant n$, $\kappa$ runs over all the sequences of $\set{1,\dots,n}$ with the same length of $\alpha,\beta$. By successively applying Eq.(\ref{bc}), we have
\begin{align*}
  \det \tensor{S}_k[\alpha,\beta]&=\det (\tensor{D}_k\tensor{C}_k\tensor{S}_{k-1}\tensor{B}_k\tensor{A}_k)[\alpha,\beta]\\
  &=\sum_{\kappa_1,\kappa_2,\kappa_3,\kappa_4} \det \tensor{D}_k[\alpha,\kappa_1] \det \tensor{C}_k[\kappa_1,\kappa_2]  \det \tensor{S}_{k-1}[\kappa_2,\kappa_3]\det\tensor{B}_k[\kappa_3,\kappa_4] \det \tensor{A}_k[\kappa_4,\beta].
\end{align*}
Together with Claim 1 and by the mathematical induction, we have $\det \tensor{S}_k[\alpha,\beta]\geqslant 0$.

Claim 3: $\det\tensor{S}_k>0$. Clearly, $\det\tensor{S}_0>0$. Since $\det \tensor{A}_k,\,\det \tensor{D}_k>0$, we need only prove that $\det (\tensor{C}_k\tensor{S}_{k-1}\tensor{B}_k)>0$. By the Binet-Cauchy formula, we have
\begin{align*}
   \det (\tensor{C}_k\tensor{S}_{k-1}\tensor{B}_k)=\sum_{i,j} \det \tensor{C}_k(|i)\det\tensor{S}_{k-1}(i|j)\det \tensor{B}_k(j|).
\end{align*}
Clearly, $\det \tensor{C}_k(|i),\,\det \tensor{B}_k(j|)>0$. By the induction assumption $\det\tensor{S}_{k-1}\neq 0$, Claim 2 implies that there exists at least one $\det\tensor{S}_{k-1}(i|j)>0$. This ends the proof.
\end{proof}



\end{document}